\documentclass[a4paper,reqno,12pt]{amsart}
 \usepackage[foot]{amsaddr}
\usepackage{amsmath,amssymb}
\usepackage{graphicx, enumerate, enumitem, color}
\usepackage[usenames,dvipsnames,svgnames,table]{xcolor}
\usepackage{xparse}
\usepackage{tikz}
\usetikzlibrary{decorations.markings,arrows}
\usepackage[utf8]{inputenc}

\colorlet{mylinkcolor}{violet}
\colorlet{mycitecolor}{YellowOrange}
\colorlet{myurlcolor}{Aquamarine}
\usepackage{hyperref}
\usepackage{xcolor}
\hypersetup{
    pdftitle={Matroid Secretaries},
    pdfauthor={},
    colorlinks,
    linkcolor={red!50!black},
    citecolor={blue!50!black},
    urlcolor={blue!80!black}
}

\makeatletter
\newtheorem*{rep@theorem}{\rep@title}
\newcommand{\newreptheorem}[2]{%
\newenvironment{rep#1}[1]{%
 \def\rep@title{#2 \ref{##1}}%
 \begin{rep@theorem}}%
 {\end{rep@theorem}}}
\makeatother

\newreptheorem{theorem}{Theorem}
\newreptheorem{lemma}{Lemma}

\newtheorem{theorem}{Theorem}
\newtheorem{conjecture}[theorem]{Conjecture}

\newtheorem{corollary}[theorem]{Corollary}

\newtheorem{lemma}[theorem]{Lemma}
\newtheorem{hypothesis}{Hypothesis}
\theoremstyle{remark}

\let\le\leqslant
\let\ge\geqslant
\let\leq\leqslant
\let\geq\geqslant

\DeclareMathOperator{\opt}{OPT}

\DeclareMathOperator{\dist}{dist}
\DeclareMathOperator{\cl}{cl}

\newcommand{\cM}{\mathcal{M}}
\def\be{\begin{equation}}
\def\ee{\end{equation}}

\newcommand{\cB}{\mathcal{B}}

\newcommand{\ev}{\mathbf{E}}
\newcommand{\pr}{\mathbf{P}}
\newcommand{\bR}{\mathbb{R}}
\newcommand{\nnr}{\mathbb{R}_{\ge 0}}
\newcommand{\alg}{{\mathsf{ALG}}}
\newcommand{\cX}{\mathcal{X}}
\newcommand{\con}{/}
\newcommand{\del}{\backslash}

\makeatletter

\makeatletter

\title[Matroid Secretaries]{The matroid secretary problem for minor-closed classes and random matroids}

\makeatletter
\let\old@setaddresses\@setaddresses
\def\@setaddresses{\bigskip\bgroup\parindent 0pt\let\scshape\relax\old@setaddresses\egroup}
\makeatother

\author[T.~Huynh]{Tony Huynh}
\address[T.~Huynh]{Department of Mathematics \\
  Universit\'e Libre de Bruxelles\\
  Belgium}
\email{tony.bourbaki@gmail.com}
\thanks{Tony Huynh is supported by ERC grant \emph{FOREFRONT} (grant agreement no.\ 615640) funded by the European Research Council under the EU's 7th Framework Programme (FP7/2007-2013).}

\author[P.~Nelson]{Peter Nelson}
\address[P.~Nelson]{Department of Combinatorics and Optimization\\ University of Waterloo \\ Canada}
\email{apnelson@uwaterloo.ca}

\begin{document}
\begin{abstract}
We prove that for every proper minor-closed class $\mathcal{M}$ of $\mathbb{F}_p$-representable matroids, there exists an $O(1)$-competitive algorithm for the matroid secretary problem on $\mathcal{M}$.  This result relies on the extremely powerful matroid minor structure theory being developed by Geelen, Gerards and Whittle.  

We also note that, for asymptotically almost all matroids, the matroid secretary algorithm that selects a random basis, ignoring weights, is $(2+o(1))$-competitive.  In fact, assuming the conjecture that almost all matroids are paving, there is a $(1+o(1))$-competitive algorithm for almost all matroids.  
\end{abstract}

\maketitle
\section{Introduction}
The \emph{secretary problem}~\cite{Lindley61, Dynkin63, Ferguson89} is a classic online selection problem.  The setup is as follows. A set of $n$ `secretaries' is presented to us in a uniformly random order.  When a secretary is presented to us, we learn his \emph{weight} (a non-negative real number) and we must make an irrevocable decision to hire him or not.  Once a secretary is hired, the process ends.  The goal is to design an algorithm which performs well \emph{in expectation}.  It is well-known that the algorithm that samples and rejects the first $\frac{1}{e}$ (here, $e$ is the base of the natural logarithm) of the secretaries and then hires the first secretary that is better than all secretaries in the sample will hire the best secretary with probability $\frac{1}{e}$. This is asymptotically the best possible~\cite{Dynkin63}.

The \emph{matroid secretary problem} is a very widely studied generalization of the secretary problem, first introduced by Babaioff, Immorlica, and Kleinberg \cite{BIK07}.  Here we are given a matroid $M$ whose elements are `secretaries' with an (unknown) weight function $w\colon E(M) \to \nnr$. The secretaries are again presented to us online in a random order.  When a secretary $x$ is presented to us, we learn its weight $w(x)$.  At this point, we must make an irrevocable decision to either hire $x$ or not.  The \emph{weight} of a set of secretaries is the sum of their weights.  Our goal is to design an algorithm that hires a set of secretaries with large weight, subject to the constraint that the hired set of secretaries is an independent set in $M$.

The \emph{expected value} of an algorithm on $(M, w)$ is the average weight of an independent set it produces, taken over all orderings of $E(M)$ and over any internal randomness of the algorithm.  We write $\opt(M,w)$ for the maximum weight of an independent set of $M$.  For $c \ge 1$, an algorithm is \emph{$c$-competitive} for $M$ if the expected value it outputs for $(M,w)$ is at least $\frac{1}{c}\opt(M,w)$ for all weight functions $w\colon E(M) \to \nnr$.  The following deep conjecture remains open. 

\begin{conjecture}[Babaioff, Immorlica, and Kleinberg \cite{BIK07}] \label{mainconjecture}
For all matroids $M$, there is a $c$-competitive matroid secretary algorithm for $M$, where $c$ is a universal constant independent of $M$.  
\end{conjecture}

The original motivation for considering the matroid secretary problem was mechanism design for online auctions.  Here the `secretaries' are agents and their weights correspond to the prices they are willing to pay for auctioned items. It turns out that matroid constraints model many real world situations and they unify many previously considered models.  For example, the situation in which an auctioneer wants to sell at most $k$ items corresponds to the case that $M$ is a rank-$k$ uniform matroid.  This is known as the \emph{multiple-choice secretary problem} and was essentially completely solved by Kleinberg \cite{Kleinberg05}.  For more on the connection to mechanism design, see \cite{BIKK08}.

Although Conjecture \ref{mainconjecture} remains open, steady progress has been made in improving the competitive ratio.  Babaioff, Immorlica, and Kleinberg \cite{BIK07} gave an $O(\log r)$-competitive algorithm for all matroids $M$, where $r$ is the rank of $M$.  This was improved to an $O(\sqrt{\log r})$-competitive algorithm by Chakraborty and Lachish \cite{CL2012}.  The state-of-the-art is an $O(\log \log r)$-competitive algorithm, first obtained by Lachish \cite{Lachish14}.  Using different tools, Feldman, Svensson, and Zenklusen \cite{FMZ15} give a simpler $O(\log \log r)$-competitive algorithm for all matroids.  

On the other hand, there exist constant-competitive algorithms for restricted classes of matroids.  We now give a partial list of such matroid classes.  Graphic matroids have a $2e$-competitive algorithm due to Korula and Pál \cite{KP09}.  Cographic matroids have a $3e$-competitive algorithm due to Soto \cite{Soto13}.  Regular matroids and max-flow min-cut matroids both have $9e$-competitive algorithms due to Dinitz and Kortsarz \cite{DK14}.  Our main result is a generalization of all these aforementioned constant-competitive algorithms.  

\subsection{Our Results}
For each prime $p$, we let $\mathbb{F}_p$ denote the field with $p$ elements.
A class of matroids $\cM$ is \emph{minor-closed} if $M \in \cM$ and $N$ a minor of $M$ implies that $N$ is also in $\cM$.  A \emph{proper minor-closed} class of $\mathbb{F}_p$-representable matroids is a minor-closed class of $\mathbb{F}_p$-representable matroids that is not equal to the class of all $\mathbb{F}_p$-representable matroids.  

Our main result is that Conjecture \ref{mainconjecture} holds for \emph{every} proper minor-closed class of $\mathbb{F}_p$-representable matroids.  Note that graphic matroids, cographic matroids, regular matroids, and max-flow min-cut matroids are all examples of proper minor-closed classes of $\mathbb{F}_p$-representable matroids.  Here is our main result.

\begin{theorem}\label{main}
    Suppose that Hypothesis~\ref{structure} holds and let $p$ be a prime. Then for every proper minor-closed class $\cM$ of $\mathbb{F}_p$-representable matroids, there is a constant $c(\cM)$
    such that every matroid in $\cM$ has a $c(\cM)$-competitive matroid secretary algorithm. 
\end{theorem}

To be forthright with the reader, we stress that Theorem \ref{main} relies on a structural hypothesis communicated to us by Geelen, Gerards, and Whittle, which has not yet appeared in print.  This hypothesis is stated as Hypothesis~\ref{structure}.  The proof of Hypothesis~\ref{structure} will stretch to hundreds of pages, and will be a consequence of their decade-plus `matroid minors project'. \footnote{We are aware of the problematic family of dyadic matroids recently discovered by Grace and van Zwam~\cite{Gv18}, but they are not counterexamples to Hypothesis~\ref{structure}.} This is a body of work generalising Robertson and Seymour's graph minors structure theorem \cite{rs03} to matroids representable over a fixed finite field, leading to a solution of Rota's Conjecture~\cite{GGW14}. See \cite{Geelen2015} for a discussion of the project.  

In proving Theorem \ref{main}, we develop techniques to handle each of the ingredients that appear in the statement of  Hypothesis~\ref{structure}. We believe these techniques may be of independent interest.  Also, our proof of Theorem \ref{main} is one of the few instances of using deep results from matroid theory in attacking Conjecture \ref{mainconjecture}.

To complement our main result, we also show that Conjecture \ref{mainconjecture} holds for almost all matroids.  For each $M$, let $\mathsf{RB}$ denote the matroid secretary algorithm that chooses a uniformly random basis $B$ of $M$, and selects precisely the elements of $B$, ignoring weights. 

\begin{theorem}\label{main2}
    For asymptotically almost all matroids on $n$ elements, the algorithm $\mathsf{RB}$ is $(2+\tfrac{3}{\sqrt{n}})$-competitive.
\end{theorem}

The veracity of Theorem~\ref{main2} is not terribly surprising, although many properties which `obviously' hold for almost all matroids are very difficult to establish.  For example, the fact that almost all matroids are $k$-connected (for each fixed $k$) was only proved recently \cite{Pv18}. Indeed, the proof of Theorem~\ref{main2} relies on recent breakthrough results of Pendavingh and van der Pol \cite{PP15, Pv18}. 

One can think of Theorem~\ref{main} and Theorem~\ref{main2} as complementary results.  Theorem~\ref{main2} asserts that the `typical' matroid has a constant-competitive algorithm, so we expect counterexamples to Conjecture \ref{mainconjecture} to come from structured classes of matroids.  On the other hand, Theorem~\ref{main} says that restricting to proper-minor closed classes of $\mathbb{F}_p$-representable matroids also cannot produce counterexamples to Conjecture \ref{mainconjecture}.

Note that for almost all matroids, the algorithm $\mathsf{RB}$ has a better competitive ratio than the best algorithm for the classical secretary problem.  Indeed, under a widely believed conjecture in asymptotic matroid theory, we note that there is in fact a $(1+o(1))$-competitive algorithm for almost all matroids. 

\begin{theorem} \label{main3}
If asymptotically almost all matroids are paving, then there is a $(1+o(1))$-competitive matroid secretary algorithm for asymptotically almost all matroids.
\end{theorem}

\subsection{Our Methods}
Roughly,  Hypothesis~\ref{structure} asserts that every matroid in a proper minor-closed class of $\mathbb{F}_p$-representable matroids admits a tree-like decomposition into pieces that are `close' to being `sparse' or `graph-like'.  It is quite interesting that many of the ingredients of Hypothesis~\ref{structure} have already been considered in the literature.  For example, it is already known that `sparse' and `graph-like' matroids (definitions will be given later) have constant-competitive algorithms \cite{Soto13}.  Moreover, 
 Dinitz and Kortsarz \cite{DK14} have a very nice framework for obtaining constant-competitive algorithms for matroids decomposable as $1$-, $2$-, and $3$-sums from matroids for which we already have constant-competitive algorithms.   
 
 In Theorem \ref{treealgorithm}, we extend this framework to allow tree-decompositions of any fixed `thickness'.  This is a non-trivial extension, and is necessary since Hypothesis~\ref{structure} requires `thickness' greater than $3$.  We also prove that with an appropriate definition of `closeness', if a minor-closed class $\cM'$ is close to another minor-closed class $\cM$ for which we already have a constant-competitive algorithm, then there is a constant-competitive algorithm for $ \cM'$ (with a worse competitive ratio).  Both of these results are independent of any matroid structure theory results.  For example, plugging in the regular matroid decomposition theorem of Seymour \cite{Seymour80} into Theorem \ref{treealgorithm}, we recover the constant-competitive algorithm for regular matroids.   Plugging in Hypothesis \ref{structure} into Theorem \ref{treealgorithm} gives Theorem \ref{main}.

\subsection{Limitations and Related Results}
Probably the greatest drawback of the matroid secretary algorithm in Theorem~\ref{main} is that it requires knowledge of the entire matroid upfront.  This is the same model as originally introduced by Babaioff, Immorlica, and Kleinberg \cite{BIK07}, but many of the known matroid secretary algorithms work in the setting where the matroid structure is also revealed online, see \cite{HKC04}.   It is a nice open problem whether Theorem~\ref{main} holds in this setting.  

We also note that the existence of constant-competitive algorithms for transversal matroids~\cite{DP12, KRTV13} and laminar matroids~\cite{IW11, HP13, JSZ13} is not implied by Theorem \ref{main}.  Laminar matroids are a minor-closed class, and every laminar matroid is representable over a sufficiently large field (see \cite{FO17}), but there is no single finite field over which all laminar matroids are representable.  The following strengthening of our main result would also imply a constant-competitive algorithm for laminar matroids. 

\begin{conjecture} \label{conj:minor-closed}
For every proper minor-closed class of matroids $\mathcal M$, there exists a constant $c(\mathcal M)$ such that every matroid in $\cM$ has a $c(\cM)$-competitive matroid secretary algorithm. 
\end{conjecture}

This is by no means an exhaustive list of related results. We refer the reader to the survey of Dinitz \cite{Dinitz13} for more information on the matroid secretary problem. 
 
\subsection{Outline of Paper}

Section~\ref{sec:basics} provides a quick review of matroid terminology and notation. In Section \ref{sec:prelim}, we describe the two `basic' classes of matroids that appear in Hypothesis~\ref{structure}.  In Section \ref{sec:closeness} we consider a notion of `closeness' of matroids, and show that constant-competitive algorithms are preserved for matroids that are `close'.  In Section \ref{sec:trees}, we introduce our definition of tree-decompositions for matroids and show that we can obtain constant-competitive algorithms for matroids that have a bounded `thickness' tree-decomposition into pieces for which we already have constant-competitive algorithms.  In Section \ref{sec:regular}, we quickly re-derive the constant-competitive algorithm for regular matroids. We prove our main result (Theorem \ref{main}) in Section \ref{sec:main}.  Finally, in Section \ref{sec:random}, we prove our results on random matroids. 

\section{Matroid Basics} \label{sec:basics}

For a more thorough introduction to matroid theory we refer the reader to Oxley \cite{Oxley11}, whose notation and terminology we follow.  
A \emph{matroid} is a pair $M=(E, \mathcal{I})$ where $E$ is a finite \emph{ground set} and $\mathcal{I}$ is a collection of subsets of $E$ satisfying

\begin{itemize}
    \item $\emptyset \in \mathcal{I}$,
    \item if $I \in \mathcal{I}$ and $J \subseteq I$, then $J \in \mathcal I$, and 
    \item if $I,J \in \mathcal I$ and $|I| < |J|$, then there exists $x \in J$ such that $I \cup \{x\} \in \mathcal I$.
\end{itemize}

A set $X \subseteq E$ is \emph{independent} if $X \in \mathcal I$, and \emph{dependent} otherwise.  The maximal independent sets are \emph{bases}, and the minimal dependent sets are \emph{circuits}.  The \emph{rank function} of $M$ is the function $r_M: 2^E \to \mathbb{Z}_{\geq 0}$ where $r_M(X)$ is the size of a largest independent set contained in $X$.  The \emph{closure} of $X \subseteq E$, denoted $\cl_M(X)$, is the largest set $Y$ such that $r_M(Y)=r_M(X)$.  The \emph{dual} of a matroid $M$, denoted $M^*$, is the matroid whose bases are the complements of the bases of $M$.  We now give several examples of matroids. 

Let $G=(V,E)$ be a graph.  The \emph{cycle matroid} of $G$, denoted $M(G)$, is the matroid with ground set $E$ and circuits the edge sets of cycles of $G$.  Such matroids are called \emph{graphic matroids}. \emph{Cographic matroids}  are the duals of graphic matroids.  

Let $\mathbb F$ be a field.  An \emph{$\mathbb F$-matrix} is a matrix with entries in $\mathbb F$.  Let $A$ be an $\mathbb F$-matrix with its columns labelled by $E$.  We let $M[A]$ be the matroid with ground set $E$ whose independent sets are the subsets of $E$ corresponding to linearly independent columns of $A$.  A matroid $M$ is \emph{$\mathbb F$-representable} if $M=M[A]$ for some $\mathbb F$-matrix $A$.   It is \emph{binary} if it is $\mathbb F_2$-representable and it is \emph{regular} if it is $\mathbb F$-representable for every field $\mathbb F$.  

A \emph{rank-$r$ uniform} matroid is a matroid where all the $r$-subsets are bases.  We denote the rank-$r$ uniform with $n$ elements as $U_{r,n}$.  It is well-known exactly which fields the rank-$2$ uniform matroids are representable over.  We only require the following direction.  

\begin{lemma} \label{uniformrep}
$U_{2, k}$ is not $\mathbb{F}$-representable if $|\mathbb{F}| \leq k-2$.  
\end{lemma}

\begin{proof}
Suppose not and let $A$ be a $2 \times k$ $\mathbb{F}$-matrix which represents $U_{2,k}$. By performing row operations we may assume that $A_{1,1}=A_{2,2}=1$ and $A_{1,2}=A_{2,1}=0$.  Since every two columns are independent, $A_{1, \ell} \neq 0$ for all $\ell \in \{3, \dots, k\}$.  By column scaling, we may assume that $A_{1, \ell} =1$ for all $\ell \in \{3, \dots, k\}$.  Now note that $A_{2, \ell} \neq 0$ for all $\ell \in \{3, \dots, k\}$.  Since $|\mathbb{F} \setminus \{0\}| \leq k-3$, there exist distinct $\ell, \ell' \in \{3, \dots, k\}$ such that $A_{2, \ell}=A_{2, \ell'}$, which contradicts that every two columns of $A$ are independent.  
\end{proof}

We now describe a relation on the set of all matroids analogous to the minor relation on graphs.  
Let $M$ be a matroid with ground set $E$, and let $C$ and $D$ be disjoint subsets of $E$.  We let $M \con C \del D$ be the matroid with ground set $E \setminus (C \cup D)$ and rank function $r_{M \con C \del D}(X)=r_M(X \cup C)-r_M(C)$.  We say that $M \con C \del D$ is obtained from $N$ by \emph{contracting} $C$ and \emph{deleting} $D$.    A matroid $N$ is a \emph{minor} of a matroid $M$, written $N \preceq M$, if $N$ is isomorphic to $M \con C \del D$ for some choice of $C$ and $D$.  A \emph{restriction} of $M$ is a minor of $M$ obtained by only deleting edges.  We write $M|X$ to denote $M \del (E - X)$.

The minor relation on matroids generalizes the minor relation on graphs in the following sense.  Let $G=(V,E)$ be a graph.  It is easy to show (see 3.1.2 and 3.2.1 in~\cite{Oxley11}) that for all $x \in E$, $M(G) \del x=M(G \del x)$ and $M(G) \con x = M(G \con x)$.  To give the reader some more intuition, we now describe what the minor relation means for representable matroids.  

Let $M$ be a representable matroid and $x \in E(M)$.  It is useful to regard the elements of $M$ as points in a projective space $\mathbb{P}$.  Deleting $x$ has an obvious meaning; we simply remove the point $x$.  To contract $x$, we choose a hyperplane $H$ not containing $x$, and we project the remaining points of $M$ onto $H$ from $x$.  It is easy to show that the resulting minor is independent of the choice of $H$.  

A class $\mathcal{M}$ of matroids is \emph{minor-closed} if $M \in \mathcal M$ and $N \preceq M$ implies $N \in \mathcal M$. By the previous two paragraphs, the class of graphic matroids is minor-closed, and for every field $\mathbb{F}$, the class of $\mathbb{F}$-representable matroids is  minor-closed. It follows that the class of regular matroids is also minor-closed.  

A \emph{proper minor-closed} class of $\mathbb{F}$-representable matroids is a minor-closed class of $\mathbb{F}$-representable matroids not equal to the class of all $\mathbb{F}$-representable matroids.  In this paper, we are interested in proper minor-closed classes of $\mathbb{F}_p$-representable matroids, where $p$ is a prime.  For example, the class of regular matroids is a proper minor-closed class of binary matroids.

\section{Sparse and Graph-like Matroids} \label{sec:prelim}
We now introduce the two classes of matroids that appear as `pieces' of the tree-decomposition in Hypothesis~\ref{structure}.  Let $M$ be a matroid and $N$ be the matroid obtained from $M$ deleting loops and suppressing parallel elements.  Let $\gamma \in \mathbb{R}$.  We say that $M$ is \emph{$\gamma$-sparse} if $|X| \le \gamma r_N(X)$ for all $X \subseteq E(N)$.  To give some intuition to the reader, note that there does \emph{not} exist $\gamma \in \mathbb{R}$ such that all graphic matroids are $\gamma$-sparse. This is because the cycle matroid of the complete graph $K_n$ has $\binom{n}{2}$ elements and rank $n-1$.  On the other hand, all graphic matroids of planar graphs are $3$-sparse, since it is well-known that a planar graph on $n$ vertices contains at most $3n-6$ edges.

Soto~\cite[Theorem 5.2]{Soto13} proved that there is constant-competitive matroid secretary algorithm for $\gamma$-sparse matroids.

\begin{theorem}\label{sparse}
    If $M$ is a $\gamma$-sparse matroid, then there is a $\gamma e$-competitive matroid secretary algorithm for $M$. 
\end{theorem}

Let $k \in \mathbb{N}$.  A \emph{$k$-column sparse} matroid is a matroid isomorphic to $M[A]$, where $A$ has at most $k$ non-zero entries per column.  Note that graphic matroids are $2$-column sparse.  Soto~\cite[Theorem 5.4]{Soto13} also proved that $k$-column sparse matroids have constant-competitive matroid secretary algorithms for every fixed constant $k$.  We only need this result for $k=2$, and we call $2$-column sparse matroids \emph{represented frame matroids}.  We use the current best ratio in~\cite[Theorem 14]{STV18}.

\begin{theorem}\label{frame}
    If $M$ is a represented frame matroid, then there is a $4$-competitive matroid secretary algorithm for $M$. 
\end{theorem}

\section{Lifts and Projections} \label{sec:closeness}
Let $M$ be a matroid.  A \emph{loop} is an element $x$ such that $r_M(\{x\})=0$, and a \emph{free element}\footnote{We use \emph{free element} instead of \emph{coloop} because it is more common in the computer science literature.} is an element $x$ such that $r_{M^*}(\{x\})=0$. Note that if $x$ is a loop, no secretary algorithm can select $x$. On the other hand, if $x$ is a free element, we may assume every secretary algorithm always selects $x$.  

Let $M_1$ and $M_2$ be matroids on a common ground set $E$. We say that $M_1$ is a \emph{distance-$1$ perturbation} of $M_2$ if there is a matroid $M$ and a non-loop, non-free element $x$ of $M$  such that $\{M \con x, M \del x\} = \{M_1,M_2\}$. We say that  $M \con x$ is a \emph{projection} of $M \del x$ and $M \del x$ is a \emph{lift} of $M \con x$ (many authors call these \emph{elementary} projections/lifts). The \emph{perturbation distance} between two matroids $M,M'$ on a common ground set is the minimum $t$ for which there exists a sequence $M = M_0, M_1, \dotsc, M_t = M'$ where each $M_i$ is a distance-$1$ perturbation of $M_{i-1}$.  We write $\dist(M,M')$ for this quantity. If $M=N$, we may take $M=M_0=N$, and so $\dist(M,M)=0$.

In this section, we show that the existence of constant-competitive algorithms is robust under a bounded number of lifts/projections.  We first note that this is true for deleting elements, which can be proved by setting elements to have zero weight appropriately. 

\begin{lemma}\label{deletion}
    If there is a $c$-competitive matroid secretary algorithm for $M$, then there is a $c$-competitive algorithm for every restriction of $M$. 
\end{lemma}

\begin{lemma}\label{lift}
    Let $N$ be a lift of a matroid $M$. If there is a $c$-competitive algorithm for $M$ then there is a $\max(e,2c)$-competitive algorithm for $N$.
\end{lemma}
\begin{proof}
    Let $\alg_M$ be a $c$-competitive algorithm for $M$. Let $L$ be a matroid and $x$ be a non-loop and non-free of $L$ such that $L \con x = M$ and $L \del x = N$. Let $P$ be the parallel class of $L$ containing $x$. Note that each element in $P - \{x\}$ is a loop in $M$, and hence will never be selected by $\alg_M$.  We specify an algorithm $\alg_N$ for $N$ as follows:
    \begin{itemize}
        \item as the elements of $P - \{x\}$ are received, $\alg_N$ runs the classical secretary algorithm to select one. If $P = \{x\}$, then no element is chosen in this way.
        \item as the elements of $E(M) - P$ are received, $\alg_N$ passes them to $\alg_M$ and selects them as $\alg_M$ does. 
    \end{itemize}
    
    Let $I$ be the set of elements chosen in the first way (so $|I| \le 1$) and $J$ be the set of elements chosen in the second way. Clearly $J$ is independent in $N \con I$ and so $I \cup J$ is independent in $N$. It remains to show that $\ev(w(I \cup J)) \ge \tfrac{1}{\max(e,2c)}\opt(N,w)$. Let $B$ be a max-weight basis of $(N,w)$ and let $C$ be the unique circuit of $L$ with $\{x\} \subseteq C \subseteq B \cup \{x\}$. To analyse $\alg_N$ we distinguish two cases. 
    
    If $|C| \ge 3$, then let $C' = C-\{x\}$ and let $y$ be a minimum-weight element of $C'$. Now $w(C'-\{y\}) \ge \tfrac{1}{2}w(C')$ and $B - \{y\}$ is a basis of $M$ satisfying $w(B-\{y\}) = w(B-C') + w(C'-\{y\}) \ge w(B-C') + \tfrac{1}{2}w(C') \ge \tfrac{1}{2}w(B)$. Therefore $\opt(M,w) \ge \tfrac{1}{2}\opt(N,w)$ and so $\ev(w(J)) \ge \tfrac{1}{c} \opt(M,w) \ge \tfrac{1}{2c}\opt(N,w)$. 
    
    If $|C| < 3$, then (since $x$ is a non-loop of $L$) we have $C = \{x,x'\}$ for some $x' \in P-\{x\}$. In this case $B = \{x'\} \cup J'$ for some independent set $J'$ of $M$. Now $I$ is chosen by an $e$-competitive secretary algorithm on $P-\{x\}$, so $\ev(w(I)) \ge \tfrac{1}{e}w(x')$. Moreover, since $B$ is optimal and $\{x'\} \cup B_0$ is independent in $N$ for every basis $B_0$ of $M$, we have $w(B - \{x'\}) = \opt(M,w)$. Therefore
    \begin{align*} \ev(w(I \cup J)) &\ge \tfrac{1}{e}w(x') + \tfrac{1}{c}\opt(M,w) \\
    &= \tfrac{1}{e}w(x') + \tfrac{1}{c}w(B- \{x'\}) \\
    &\ge \tfrac{1}{\max(e,c)}\opt(N,w).
    \end{align*}
    
    It follows from these two cases that $\alg_N$ is $\max(e,2c)$-competitive for $N$.
\end{proof}

\begin{lemma}\label{multprojection}
    Let $N$ be a matroid obtained from a matroid $M$ by a sequence of $t$ projections. Let $L$ be the set of loops of $N$. If there is a $c$-competitive algorithm for $M$, then there is a $c(e+1)^t
    $-competitive algorithm for $M \del L$ whose output is always independent in $N$. 
\end{lemma}

\begin{proof}
We proceed by induction on $t$.  The case $t=0$ is trivial.  Suppose $t \geq 1$ and let $N'$ be the matroid such that $N'$ is obtained from $M$ by $t-1$ projections and $N$ is obtained from $N'$ by a single projection.  Let $L'$ and $L$ be the set of loops of $N'$ and $N$, respectively.   By induction, there is a $c(e+1)^{t-1}$-competitive algorithm for $M \del L'$ whose output is always independent in $N'$.  Since $L' \subseteq L$, by Lemma~\ref{deletion}, there is a $c(e+1)^{t-1}$-competitive algorithm $\alg_{M \del L}$ for $M \del L$ whose output is always independent in $N'$.  Let $P$ be a matroid and $x$ be a non-loop and non-free element of $P$ so that $P \con x = N$ and $P \del x = N'$. Let $X$ be the random variable taking the value $\mathsf{heads}$ with probability $\tfrac{e}{e+1}$ and $\mathsf{tails}$ with probability $\tfrac{1}{e+1}$. We define another matroid secretary algorithm $\alg_N$ for $M \del L$ as follows:
    \begin{itemize}
        \item if $X = \mathsf{heads}$, then we run a classical secretary algorithm on $N \del L$ to select just one element. 
        \item if $X = \mathsf{tails}$, then we run $\alg_{M \del L}$, except we only pretend to hire any element whose selection would create a dependency in $N$ with the elements already chosen.
    \end{itemize}
    In both cases, $\alg_N$ clearly produces an independent set in $N$.
 Fix a weighting $w$ of $M \del L$. Let $x_0 \in E(M \del L)$ have maximum weight, and let $I$ be the set selected by $\alg_N$. If $X = \mathsf{heads}$, then $\ev(w(I) | X = \mathsf{heads}) \ge \tfrac{1}{e}w(x_0)$. Moreover, if $J$ is the set output by $\alg_{M \del L}$, then $J \cup \{x\}$ contains at most one circuit of $P$.  It follows that if $X = \mathsf{tails}$ then $I$ is obtained from $J$ by removing at most one element, so $w(I) \ge w(J) - w(x_0)$. Thus
    \begin{align*}
        \ev(w(I)) &= \tfrac{e}{e+1} \ev(w(I)|X = \mathsf{heads}) + \tfrac{1}{e+1}\ev(w(I)| X = \mathsf{tails}) \\
        & \ge \tfrac{1}{e+1}w(x_0)  + \tfrac{1}{e+1}(\ev(w(J) - w(x_0)) | X = \mathsf{tails})\\
        & = \tfrac{1}{e+1}\ev(w(J) | X = \mathsf{tails})\\
        & \ge \tfrac{1}{c(e+1)^t}\opt(M \del L,w).
    \end{align*}
    It follows that $\alg_N$ is $(e+1)^tc$-competitive for $M \del L$.
\end{proof}

In particular, since every independent set in $N$ is an independent set of $M \del L$, the algorithm $\alg_N$ is $(e+1)c$-competitive for $N$. Since $(e+1)c > \max(e,2c)$ for $c \ge 1$, we can thus combine Lemmas~\ref{lift} and~\ref{multprojection} with an inductive argument to yield the following. 

\begin{lemma}\label{perturb}
Let $t \in \mathbb{N}$ and let $M$ and $N$ be matroids with $\dist(M,N) \le t$. If there is a $c$-competitive algorithm for $M$, then there is a $c(e+1)^t$-competitive algorithm for $N$.  
\end{lemma}

\section{Tree-decompositions} \label{sec:trees}
In this section, we introduce a notion of tree-decompositions of matroids, and give a constant-competitive matroid secretary algorithm for matroids having a `bounded-thickness' tree-decomposition into pieces for which constant-competitive algorithms are known. Note that there already exist various notions of tree-decompositions of matroids, such as \emph{branch-decompositions} \cite{GGW02}.  Therefore, to avoid confusion we use the term \emph{thickness} for the `width' of our tree-decompositions.

We begin by defining the connectivity function of a matroid.  
 Let $M = (E,r_M)$ be a matroid.  For disjoint sets $X,Y \subseteq E$, the \emph{local connectivity} between $X$ and $Y$ is $\sqcap_M(X,Y):= r_M(X) + r_M(Y) - r_M(X \cup Y)$. The \emph{connectivity function} $\lambda_M$ of $M$ is the function $\lambda_M\colon 2^{E} \to \mathbb{Z}_{\geq 0}$ defined by $\lambda_M(X)= \sqcap_M(X,E-X)$. 
 
 To give the reader some intuition, we now describe what the connectivity function $\lambda_M$ means for graphic and representable matroids.  Let $G=(V,E)$ be a graph and $M=M(G)$ be the cycle matroid of $G$.  For $X \subseteq E$, we let $G[X]$ be the subgraph of $G$ induced by the edges in $X$.  It is easy to show that if $G[X]$ and $G[E-X]$ are both connected, then $\lambda_{M}(X)+1$ is the number of vertices in both $G[X]$ and $G[E-X]$ (see~\cite[Lemma 8.1.7]{Oxley11}).  Thus, $\lambda_{M}(X)$ is a measure of how `connected' $X$ is to $E - X$.  Indeed, it is possible to define graph connectivity using $\lambda_M$ (see~\cite[Section 8.6]{Oxley11}).     

Let $M=M[A]$ for some matrix $A$ with column labels $E$.  For $X \subseteq E$ we let $\langle X \rangle$ be the subspace generated by the columns corresponding to $X$. In this case, $\lambda_M(X)$ is the dimension of $\langle X \rangle \cap \langle E-X \rangle$.  So again, $\lambda_M(X)$ is a measure of how `connected' $X$ is to $E-X$.     

 Next, we show that the connectivity function $\lambda_M$ is related to the perturbation distance between deletion and contraction minors of $M$.  
 
 \begin{lemma}\label{lambda}
     Let $M = (E,r_M)$ be a matroid and let $X \subseteq E$. Then $M \con (E-X)$ is obtained from $M|X$ by a sequence of $\lambda_M(X)$ projections. 
 \end{lemma}
 \begin{proof}
   Let $I_1 \subseteq X$ and $I_2,I_3 \subseteq E-X$ be disjoint independent sets of $M$ so that $I_1$ is a basis for $M|X$, $I_1 \cup I_2$ is a basis for $M$, and $I_2 \cup I_3$ is a basis for $M \del X$. We have $|I_3| = r_M(E-X) - (r_M(E) - r_M(X)) = \lambda_M(X)$. Now the matroid $N = (M \con I_2)|(X \cup I_3)$ satisfies $N \con I_3 = M \con (E-X)$ and $N \del I_3 = M|X$. The lemma follows.
 \end{proof}
 
 We are now ready to introduce our notion of tree-decompositions.  A \emph{tree-decomposition} of a matroid $M$ is a pair $(T, \mathcal{X})$ where $T$ is a tree and $\mathcal{X}:=\{X_v : v \in V(T)\}$ is a  partition of $E(M)$ indexed by $V(T)$.
Let $e = v_1v_2 \in E(T)$ and $T_1$ and $T_2$ be the components of $T \del e$ where $v_i \in V(T_i)$.  Let $X_1:=\bigcup_{v \in V(T_1)} X_v$.  We define the \emph{thickness} of $e$, $\lambda(e)$, to be $\lambda_M(X_1)$.  The \emph{thickness} of $(T, \mathcal{X})$ is $\max_{e \in E(T)} \lambda(e)$.  If, for all $e = uv \in E(T)$, we have $\sqcap_M(X_u,X_v) = \lambda(e)$, then we say $(T,\mathcal{X})$ is a \emph{full tree-decomposition} of $M$. For those familiar with the standard definition of tree-decompositions of graphs (see~\cite[Section 12.3]{Diestel11}), fullness is analogous to the fact that the separations displayed by a standard tree-decomposition of a graph are obtained by intersecting adjacent bags of the tree~\cite[Lemma 12.3.1]{Diestel11}.

\begin{theorem} \label{treealgorithm}
Let $\mathcal{M}$ be a class of matroids for which there exists a $c$-competitive matroid secretary algorithm. Let $k \in \mathbb{N}$ and let $t_k(\mathcal{M})$ be the set of all matroids $M$ with a full tree-decomposition $(T, \mathcal{X})$ of thickness at most $k$ such that  $M| \cl_M(X_v) \in \mathcal{M}$ for each $v \in V(T)$.  Then there is an $c(e+1)^k$-competitive matroid secretary algorithm for $t_k(\mathcal{M})$.  
\end{theorem}
\begin{proof}
    We say that a tree-decomposition $(T,\cX)$ of a matroid $M$ is an \emph{$\cM$-tree decomposition} if $M|\cl_M(X_v) \in \cM$ for all $v \in V(T)$. For each $m \ge 1$, let $t_{k,m}(\cM)$ denote the class of matroids in $t_k(\cM)$ having a full $\cM$-tree-decomposition $(T,\cX)$ of thickness at most $k$ with $|V(T)| \le m$.  There is clearly a $c(e+1)^k$-competitive matroid secretary algorithm for every matroid in $t_{k,1}(\cM) = \cM$. Let $m > 1$ and suppose inductively that every matroid in $t_{k,m-1}(\cM)$ has a $c(e+1)^k$-competitive matroid secretary algorithm.
    
    Let $M \in t_{k,m}(\cM)$, let $E = E(M)$, and let $(T,\cX)$ be a full $\cM$-tree-decomposition of $M$ of thickness at most $k$ with $|V(T)| \le m$.
    We may assume that $|V(T)|=m$, else we are done by the induction hypothesis. 
    Let $\ell$ be a leaf of $T$ and let $e = \ell u$ be the edge of $T$ incident with $\ell$. Let $X_\ell'$ and $X_u'$ be obtained from $X_\ell$ and $X_u$ by moving all elements from $X_\ell \cap \cl_M(X_u)$ into $X_u$.  It is easy to check that $\lambda_M(X_\ell') \le \lambda_M(X_\ell)$. Since $\cl_M(X_u')=\cl_M(X_u)$, we also have $r_M(E-X_\ell)=r_M(E-X_\ell')$.  Therefore,
    
    \begin{align*}
         \sqcap_M(X_\ell', X_u') &= r_M(X_\ell')+r_M(X_u')-r_M(X_\ell' \cup X_u') \\
         &= r_M(X_\ell')+r_M(X_u)-r_M(X_\ell \cup X_u) \\
         &= r_M(X_\ell')-r_M(X_\ell)+\sqcap(X_\ell, X_u)\\
         &= r_M(X_\ell')-r_M(X_\ell)+\lambda_M(X_\ell) \\
         &= r_M(X_\ell')+r_M(E-X_\ell)-r_M(E) \\
         &= r_M(X_\ell')+r_M(E-X_\ell')-r_M(E) \\
         &= \lambda_M(X_\ell').
     \end{align*}
    
    Again since $\cl_M(X_u')=\cl_M(X_u)$, it follows that $\sqcap_M(X_u',X_v)=\sqcap_M(X_u,X_v)$ for all $v \notin \{u,\ell\}$. Therefore, this move preserves the property of being a full $\cM$-tree-decomposition of thickness at most $k$; and so we may assume that $X_\ell \cap \cl_M(X_u) = \varnothing$. 
    
    Let $M(\ell) = M|X_\ell$ and let $M'(\ell) = M \con (E-X_{\ell})$. By Lemma~\ref{lambda}, the latter is obtained from the former by at most $\lambda_M(X_\ell) \le k$ projections. Moreover, since $\sqcap_M(X_\ell,X_u) = \lambda_M(X_\ell)$, we have $\lambda_{M \con X_u}(X_\ell) = 0$ and so $M'(\ell) = (M \con X_u)|X_\ell$; since $X_\ell \cap \cl_M(X_u) = \varnothing$ it follows that $M'(\ell)$ has no loops.  
    
     By Lemmas~\ref{deletion} and~\ref{multprojection}, there is a $c(e+1)^k$-competitive algorithm $\alg_{\ell}$ for $M(\ell)$ whose output is always independent in $M'(\ell)$. Moreover, we see that $(T \del \ell,\{X_w\colon w \in V(T \del \ell)\})$ is a full $\cM$-tree-decomposition of $M \del X_\ell$ with thickness at most $k$, so $M \del X_\ell \in t_{k,m-1}(\cM)$ and there is thus a $c(e+1)^k$-competitive algorithm $\alg'$ for $M \del X_\ell$. Define an algorithm $\alg$ for $M$ by running $\alg'$ and $\alg_\ell$ on the elements of $E-X_\ell$ and $X_\ell$ respectively as they are received, choosing all elements chosen by either. 
     
     Let $I_\ell$ and $I'$ be the sets chosen by $\alg_\ell$ and $\alg'$ respectively. Since $I'$ is independent in $M \del X_\ell$ and $I_{\ell}$ is independent in $M'(\ell)$, the set $I = I_{\ell} \cup I'$ chosen by $\alg$ is independent in $M$. Moreover, for each weighting $w$ of $M$ we have 
     
     \begin{align*}
         \ev(w(I)) &= \ev(w(I')) + \ev(w(I_\ell)) \\
         &\ge \tfrac{1}{c(e+1)^k}(\opt(M(\ell),w|X_\ell) + \opt(M \del X_\ell,w|(E - X_\ell)) \\
         &\ge \tfrac{1}{c(e+1)^k}\opt(M,w),
     \end{align*}
     since each independent set of $M$ is the union of an independent set of $M(\ell)$ and one of $M \del X_\ell$. The theorem follows. 
    \end{proof}

\section{Regular Matroids} \label{sec:regular}  
As an easy corollary of Theorem \ref{treealgorithm}, we obtain a short proof that there is an $O(1)$-competitive algorithm for regular matroids, a result first proved by  Dinitz and Kortsarz \cite{DK14}. The constant $9e$ that they obtain is better than ours by a factor of $\tfrac{1}{3}(e+1)^2 \approx 4.6$.

\begin{corollary} \label{regular}
There is a $3e(e+1)^2$-competitive matroid secretary algorithm for each regular matroid.  
\end{corollary}

\begin{proof}
By Seymour's regular matroid decomposition theorem \cite{Seymour80}, every regular matroid $M$ is obtained from pieces that are either graphic, cographic, or $R_{10}$ by $1$-, $2$- or $3$-sums. This gives a tree-decomposition $(T,\cX)$ of thickness at most $2$ in $M$ so that each $M|X_v$ is either graphic, cographic or $R_{10}$. Moreover, by performing parallel extensions of the elements to be deleted before each $2$-sum and $3$-sum, one can construct a matroid $M'$ having $M$ as a restriction and a \emph{full} tree-decomposition $(T,\cX')$ of $M'$ so that each $M'|\cl_{M'}(X'_v)$ is either graphic, cographic or a parallel extension of $R_{10}$. 

Korula and P{\'a}l \cite{KP09} proved that there is a $2e$-competitive matroid secretary algorithm for graphic matroids.  Soto \cite{Soto13} proved that there is a $3e$-competitive matroid secretary algorithm for cographic matroids. Since $R_{10}$ is the union of two bases, Theorem~\ref{sparse} implies that each of its parallel extensions has a $2e$-competitive algorithm. It follows from Theorem~\ref{treealgorithm} that there is a $3e(e+1)^2$-competitive algorithm for $M'$; by Lemma~\ref{deletion} there is also one for $M$.
\end{proof}

\section{Proper Minor-Closed Classes} \label{sec:main}

In this section we show that, contingent on a certain deep structural hypothesis, there is a constant-competitive algorithm for every proper minor-closed class of matroids representable over a fixed prime field.  As well as this structural hypothesis, we require one other result, due to Geelen (\cite{Geelen2011}, Theorem 4.3). We denote the matroid with $n$ elements where all $k$-subsets are independent as $U_{k,n}$.

\begin{theorem}\label{excludeclique}
    Let $q \geq 2$ and $n$ be positive integers. If $M$ is a simple matroid with no $U_{2,q+2}$-minor and no $M(K_n)$-minor, then $|M| \le q^{q^{3n}}r(M).$
\end{theorem}

Since no restriction of such an $M$ has either of the forbidden minors, for such an $M$ we also have $|X| \le q^{q^{3n}}r_M(X)$ for each $X \subseteq E(M)$. Combining this theorem with Theorem~\ref{sparse}, we thus have the following.

\begin{corollary}\label{secexcludeclique}
    Let $q \geq 2$ and $n$ be positive integers. If $M$ is a matroid with no $U_{2,q+2}$-minor and no $M(K_n)$-minor, then there is an $eq^{q^{3n}}$-competitive matroid secretary algorithm for $M$. 
\end{corollary}

We now state the structural hypothesis we will use to prove Theorem~\ref{main}. 

\begin{hypothesis}\label{structure}
Let $p$ be a prime. For every proper minor-closed class $\cM$ of $\mathbb{F}_p$-representable matroids, there exist $k,n$ and $t$ for which every $M \in \mathcal{M}$ is a restriction of an $\mathbb{F}_p$-representable matroid $M'$ having a full tree-decomposition $(T, \mathcal{X})$ of thickness at most $k$ such that for all $v \in V(T)$, if $M'|\cl_{M'}(X_v)$ has an $M(K_n)$-minor, then there is a represented frame matroid $N$ with $\dist(M'|\cl_{M'}(X_v),N) \le t$.
\end{hypothesis}

We can now prove Theorem~\ref{main}, which we restate here. 
\begin{reptheorem}{main}
    Suppose that Hypothesis~\ref{structure} holds. If $p$ is prime and $\cM$ is a proper minor-closed subclass of the $\mathbb{F}_p$-representable matroids, then there exists $c = c(\cM)$ so that every $M \in \cM$ has a $c$-competitive matroid secretary algorithm. 
\end{reptheorem}
\begin{proof}
Let $k,n$ and $t$ be the integers given for $\cM$ by Hypothesis~\ref{structure}. Let  $\cM_1$ denote the class of matroids having perturbation distance at most $t$ from some represented frame matroid. Let $\cM_2$ denote the class of matroids with no $U_{2,p+2}$-minor or $M(K_n)$-minor. By Theorem~\ref{frame} and Lemma~\ref{perturb}, every matroid in $\cM_1$ has a $4(e+1)^t$-competitive matroid secretary algorithm. Corollary~\ref{secexcludeclique} gives an $ep^{p^{3n}}$-competitive matroid secretary algorithm for each matroid in $\cM_2$. 

By Lemma~\ref{uniformrep}, $U_{2,p+2}$ is not $\mathbb{F}_p$-representable.  Thus, by Hypothesis~\ref{structure}, every $M \in \cM$ is a restriction of a matroid $M'$ having a proper tree-decomposition $(T,\cX)$ of thickness at most $k$ so that $M|\cl_{M'}(X_v)  \in \cM_1 \cup \cM_2$ for each $v \in V(T)$. By Theorem~\ref{treealgorithm}, every such $M'$ has a $c$-competitive matroid secretary algorithm, where $c = (e+1)^k\max\left(4(e+1)^t,ep^{p^{3n}}\right)$. By Lemma~\ref{deletion}, the same is true for every $M \in \cM$. 
\end{proof}

\section{Asymptotic Results} \label{sec:random}

    We say \emph{asymptotically almost all matroids} have a property if the proportion of matroids with ground set $\{1, \dotsc, n\}$ having the property tends to $1$ as $n$ approaches infinity.  We finish our paper by showing that asymptotically almost all matroids have a constant-competitive matroid secretary algorithm.  We require two recent results of Pendavingh and van der Pol \cite{PP15, Pv18}. 

\begin{theorem}[Theorem 1.3 from \cite{Pv18}] \label{manybases}
    There exists $\alpha > 0$ so that asymptotically almost all rank-$r$ matroids on $n$ elements have at least $\left(1 - \tfrac{\alpha (\log n)^3}{n}\right)\binom{n}{r}$ bases.
\end{theorem}
\begin{theorem}[Theorem 16 from \cite{PP15}] \label{rankbounds}
    If $\beta > \sqrt{\tfrac{1}{2}\ln(2)}$, then asymptotically almost all matroids on $n$ elements have rank between $\tfrac{n}{2} - \beta \sqrt{n}$ and $\tfrac{n}{2} + \beta \sqrt{n}$. 
\end{theorem}

    Recall that $\mathsf{RB}$ denotes the algorithm that selects a basis $B$ uniformly at random from the set of all bases of $M$, and chooses the elements of $B$ as secretaries regardless of the weights. By Theorem~\ref{manybases}, nearly every $r$-set in a typical matroid is a basis, so a uniformly random basis can be sampled in probabilistic polynomial time in almost all matroids by repeatedly choosing a uniformly random $r$-set until a basis is chosen.

    We now state and prove a stronger version of Theorem~\ref{main2}.

\begin{reptheorem}{main2}
    Let $\gamma > \sqrt{8 \ln(2)}$. For asymptotically almost all matroids $M$ on $n$ elements, the algorithm $\mathsf{RB}$ is $(2+ \tfrac{\gamma}{\sqrt{n}})$-competitive for $M$. 
\end{reptheorem}
\begin{proof}
    Let $\gamma' \in (\sqrt{8 \ln (2)},\gamma)$. Note that for each $\alpha \in \bR$ we have $\tfrac{1}{2} - \tfrac{\gamma'}{4 \sqrt{n}} - \tfrac{\alpha (\log n)^3}{n} > (2 + \tfrac{\gamma}{\sqrt{n}})^{-1}$ for all sufficiently large $n$. Let $n$ be a positive integer and let $M$ be a matroid on $n$ elements with ground set $E = \{1,\dotsc, n\}$ and rank $r$. Let $\mathcal{B} \subseteq \binom{E}{r}$ denote the set of bases of $M$. For each $e \in E$ let $\cB_e = \{B \in \cB: e \in B\}$. By Theorems~\ref{manybases} and~\ref{rankbounds}, there exists $\alpha \in \bR$ so that, asymptotically almost surely, $|\cB| \ge (1 - \tfrac{\alpha (\log n)^3}{n})\binom{n}{r}$ and $r \ge (\tfrac{1}{2} - \tfrac{\gamma'}{4\sqrt{n}})n$. Thus, almost surely, 
    \begin{align*}
        |\cB_e| &\ge |\{B \in \binom{E}{r}: e \in B\}| - \left(\binom{n}{r}-|\cB|\right) \\
        &= \frac{r}{n}\binom{n}{r} - \binom{n}{r} + |\cB| \\ 
        &= \binom{n}{r}\left(\frac{|\cB|}{\binom{n}{r}} + \frac{r}{n} - 1\right).\\
        &\ge |\cB|\left(1 - \tfrac{\alpha (\log n)^3}{n} + \tfrac{1}{2} - \tfrac{\gamma'}{4\sqrt{n}} - 1 \right)\\
        &\ge |\cB|(2 + \tfrac{\gamma}{\sqrt{n}})^{-1}. 
    \end{align*}
    Let $B_0$ be a basis chosen uniformly at random from $\cB$, as per $\mathsf{RB}$. For each $e \in E$, we have $\pr(e \in B_0) = |\cB_e|/|\cB| \ge (2 + \tfrac{\gamma}{\sqrt{n}})^{-1} $. Given a weighting $w$ of $M$, we therefore have \[\ev(w(B_0)) = \sum_{e \in E}w(e) \pr(e \in B_0) \ge (2 + \tfrac{\gamma}{\sqrt{n}})^{-1} w(E)\ge (2 + \tfrac{\gamma}{\sqrt{n}})^{-1}\opt(M,w).\]  $\mathsf{RB}$ is thus $(2 + \tfrac{\gamma}{\sqrt{n}})$-competitive.
\end{proof}

A rank-$r$ matroid is $\emph{paving}$ if all its circuits have cardinality at least $r$. Although few well-known constructions of matroids have this property, it is believed to almost always hold.  The following conjecture due to Mayhew, Newman, Welsh, and Whittle \cite{MNWW11} is central in asymptotic matroid theory.

\begin{conjecture}\label{pavingconjecture}
    Asymptotically almost all matroids on $n$ elements are paving. 
\end{conjecture}

An alternative characterisation is that a matroid $M$ is paving if and only if its truncation to rank $r(M)-1$ (which we denote as $T(M)$) is a uniform matroid. The matroid secretary problem for uniform matroids is known as the \emph{multiple-choice secretary problem}, and has been essentially completely solved by Kleinberg \cite{Kleinberg05}. In the language of matroids, his result is the following.

\begin{theorem}
    Let $r > 25$. There is an algorithm $\mathsf{UNI}$ that is  $(1 - \tfrac{5}{\sqrt{r}})^{-1}$-competitive  for each uniform matroid $U_{r,n}$. 
\end{theorem}

We slightly modify $\mathsf{UNI}$ to deal with paving matroids. Let $\mathsf{PAV}$ be the algorithm that, given a rank-$(r > 26)$ weighted paving matroid $(M,w)$, returns the output of $\mathsf{UNI}$ on the rank-$(r-1)$ weighted uniform matroid $(T(M),w)$. Note that, since every independent set in $T(M)$ is independent in $M$, this output is always legal. 

\begin{theorem} \label{paving}
    The algorithm $\mathsf{PAV}$ is $(1 - \tfrac{6}{\sqrt{r}})^{-1} $-competitive for every paving matroid of rank at least $37$. 
\end{theorem}
\begin{proof}
    Let $B$ be a maximum-weight basis of a rank-$r$ paving matroid $M$ with weights $w$, and let $x \in B$ have minimum weight. Now $B-\{x\}$ is a basis of $T(M)$ and has weight at least $\tfrac{r-1}{r}w(B)$, so $\opt(T(M),w) \ge (1 - \tfrac{1}{r})\opt(M,w)$. Let $I$ be the independent set output by $\mathsf{UNI}$ on $T(M)$. Now 
    \[\ev(w(I)) \ge (1 - \tfrac{5}{\sqrt{r-1}})\opt(T(M),w) \ge (1-\tfrac{5}{\sqrt{r-1}})(1-\tfrac{1}{r})\opt(M,w).\]
    Since $(1- \tfrac{5}{\sqrt{r-1}})(1-\tfrac{1}{r}) \ge 1 - \tfrac{6}{\sqrt{r}}$ for all $r \ge 37$,  $\mathsf{PAV}$ is thus $(1 - \tfrac{6}{\sqrt{r}})^{-1}$-competitive. 
\end{proof}

As promised, we finish with a proof of Theorem \ref{main3}, which we restate for convenience. 

\begin{reptheorem}{main3}
If asymptotically almost all matroids are paving, then there is a $(1+o(1))$-competitive matroid secretary algorithm for asymptotically almost all matroids.
\end{reptheorem}

\begin{proof}
By Theorem~\ref{rankbounds}, almost all matroids have $37 < r \approx \tfrac{n}{2}$ and so, conditioned on Conjecture~\ref{pavingconjecture}, Theorem \ref{paving} implies that $\mathsf{PAV}$ is $(1+ \epsilon(n))$-competitive for asymptotically almost all matroids on $n$ elements, where $\epsilon(n) \approx \tfrac{6\sqrt{2}}{\sqrt{n}}$.
\end{proof}

\bibliography{matroidsecretary}
\bibliographystyle{plain}

\end{document}